\documentclass[11pt, a4paper]{amsart}
\usepackage[utf8]{inputenc}

\usepackage{dsfont}
\usepackage{polynom}
\usepackage{empheq}
\usepackage{subfiles}
\usepackage{systeme}
\usepackage{shuffle}
\usepackage{ytableau}
\usepackage{CJK}
\usepackage{bm}
\usepackage{faktor}
\usepackage{ifthen}
\usepackage{mathrsfs}
\usepackage{amsmath}
\usepackage{mathtools}
\usepackage{amssymb}
\usepackage{amscd}
\usepackage{appendix}
\usepackage{amsthm}
\usepackage{amsfonts}
\usepackage[all]{xy}
\usepackage{array}
\usepackage{enumerate}
\usepackage{graphicx}
\usepackage{hyperref}
\usepackage{longtable}
\usepackage[OT2,T1]{fontenc}
\usepackage{verbatim}
\usepackage{tikz}
\usetikzlibrary{matrix}
\usepackage{url,setspace,multirow,tikz-cd}
\DeclareSymbolFont{cyrletters}{OT2}{wncyr}{m}{n}
\DeclareMathSymbol{\Sha}{\mathalpha}{cyrletters}{"58}
\usepackage{xcolor}
\usepackage{textcomp}
\usepackage{manyfoot}
\usepackage{booktabs}
\usepackage{algorithm}
\usepackage{algorithmicx}
\usepackage{algpseudocode}
\usepackage{listings}
\usepackage{mathdots}

\theoremstyle{thmstyleone}
\newtheorem{thm}{Theorem}[section]

\newtheorem{lem}[thm]{Lemma}
\newtheorem{cor}[thm]{Corollary}

\theoremstyle{definition}
\newtheorem{eg}[thm]{Example}
\newtheorem{rem}[thm]{Remark}
\newtheorem{ques}[thm]{Question}

\newtheorem{defn}[thm]{Definition}

\numberwithin{equation}{section}

\newcommand{\ZZ}{{\mathbb Z}}
\newcommand{\QQ}{{\mathbb Q}}

\newcommand{\NN}{{\mathbb N}}

\newcommand{\afk}{\mathfrak{a}}

\newcommand{\sfk}{\mathfrak{s}}

\newcommand{\Fcal}{\mathcal{F}}

\newcommand{\Hcal}{\mathcal{H}}

\newcommand{\Tcal}{\mathcal{T}}

\newcommand{\Dcal}{\mathcal{D}}

\newcommand{\Xcal}{\mathcal{X}}
\newcommand{\Ycal}{\mathcal{Y}}

\newcommand{\pr}{\operatorname{pr}}

\begin{document}

\title[Arborified MZV]{A map between arborifications of multiple zeta values}

\author{Ku-Yu Fan}

\address{Graduate School of Mathematics, Nagoya University, Furo-cho, Chikusa-ku, Nagoya, 464-8602, Japan.}
\email{ku-yu.fan.d2@math.nagoya-u.ac.jp }

\date{\today}


\begin{abstract}
Arborified multiple zeta values are a generalization of multiple zeta values associated with rooted trees. There are two types of decorated rooted trees, corresponding respectively to the series and the integral expressions. Manchon introduces the contracting arborification (resp. the simple arborification), which is maps from the BCK Hopf algebras of  the decorated rooted trees corresponding to the series expression (resp. the integral expression) to the non-commutative polynomial algebras of the set $\NN$ (resp. the set $\{0,1\}$). There is a natural map between the two non-commutative polynomial algebras. Manchon posed the question of finding a natural map between the two BCK Hopf algebras that would make the diagram commutative. In this paper, we consider planar rooted trees and use a recursive method to construct such a map between the two BCK Hopf algebras, making the diagram commutative.
\end{abstract}



\maketitle

\tableofcontents

\section{Introduction}\label{Introduction}
The multiple zeta values (MZVs) are the real numbers defined by the series
$$\zeta(k_1, \ldots, k_d) = \sum_{0<n_1<\cdots<n_d} \frac{1}{n_1^{k_1} \cdots n_d^{k_d}},$$
where $k_1,\ldots,k_{d-1}\in \ZZ_{>0}$ and $k_d\in \ZZ_{>1}$ for convergence of the series.

One important property of MZVs is their iterated integral representation.
$$\zeta(k_1, \ldots, k_d) = (-1)^d I(0;1,\{0\}^{k_1-1},\ldots,1,\{0\}^{k_d-1}; 1),$$
where
\begin{equation}\label{iter.}
  I(a_0; a_1, \ldots, a_k; a_{k+1}) = \int_{a_0<t_1<\cdots<t_k<a_{k+1}} \prod_{j=1}^{k}\frac{dt_j}{t_j-a_j}.
\end{equation}

Yamamoto \cite{Yamamoto2014MZV} introduced an integral associated with 2-posets, known as Yamamoto’s integral, which generalizes the iterated integral.
It is associated with the Hasse diagram of 2-posets (with black and white colorings), where the black and white vertices correspond to the differential form $\frac{dt}{1-t}$ and $\frac{dt}{t}$ respectively. For example,
\begin{equation}\label{Yamamoto’s integral}
  I \left(\begin{xy}
{(0,-2) \ar @{{*}-o} (4,2)},
{(4,2) \ar @{o-{*}} (8,-2)}
\end{xy}\right) = \int_{\substack{0<t_i<1 \\ t_1<t_2>t_3}}\frac{dt_1}{1-t_1}\frac{dt_2}{t_2}\frac{dt_3}{1-t_3}.
\end{equation}

Manchon \cite{AMZV} introduced two kinds of multiple zeta values associated with rooted trees, called arborified multiple zeta value of the first and the second kind. Arborified multiple zeta values of the first kind (cf. Definition \ref{defn arborified multiple zeta values of the first kind}) are defined for a rooted tree decorated with integers. For example,
$$\zeta\left(\begin{xy}
    {(-4,-2)*++!R{\scriptstyle k_1} \ar @{{*}-{*}} (0,2)*++!D{\scriptstyle k_2}},
    {(0,2) \ar @{{*}-{*}} (4,-2)*++!L{\scriptstyle k_3}}
  \end{xy}\right) = \sum_{\substack{0<n_i \\ n_1<n_2>n_3}} \frac{1}{n_1^{k_1}n_2^{k_2}n_3^{k_3}}.$$
Arborified multiple zeta values of the second kind (cf. Definition \ref{defn arborified multiple zeta values of the second kind}) are defined for a rooted tree decorated with 0 and 1. The example \eqref{Yamamoto’s integral} of Yamamoto’s integral is considered as an instance of the first kind, where the white vertex is treated as the root.

Manchon \cite{AMZV} considered more general $\Dcal$-decorated rooted trees which form Butcher-Connes-Kreimer Hopf algebra $\Hcal_{BCK}^\Dcal$. In the case of multiple zeta values, the set $\Dcal$ is taken to be $\Ycal \coloneqq \{y_n\,|\,n\in \NN\}$ for the first kind or $\Xcal \coloneqq \{x_0, x_1\}$ for the second kind. He also considered the simple arborification and contracting arborification, both of which are Hopf algebra morphisms (cf. Theorem \ref{thm arborification}, Definition \ref{defn arborification})
$$\afk_\Xcal : \Hcal_{BCK}^\Xcal \rightarrow \QQ\langle\Xcal\rangle \text{ and } \afk_\Ycal : \Hcal_{BCK}^\Ycal \rightarrow \QQ\langle\Ycal\rangle$$
respectively, where $\QQ\langle\Dcal\rangle$ (with $\Dcal = \Xcal$ or $\Ycal$) is the non-commutative polynomial algebra generated by $\Dcal$. Let $\sfk : \QQ\langle\Ycal\rangle \rightarrow \QQ\langle\Xcal\rangle$ be the algebra morphism defined by
$$\sfk(y_{n_1}\cdots y_{n_d}) = x_1x_0^{n_1-1}\cdots x_1x_0^{n_d-1}.$$
Manchon \cite{AMZV} posed a question of finding a natural map $\sfk^T$ with respect to the tree structures, such that the following diagram commutes:
\begin{equation}\label{diagram}
  \begin{tikzcd}
    \Hcal_{BCK}^\Ycal \arrow[r, "\sfk^T"] \arrow[d, "\afk_\Ycal"'] & \Hcal_{BCK}^\Xcal \arrow[d, "\afk_\Xcal"] \\
    \QQ\langle\Ycal\rangle \arrow[r, "\sfk"]                 & \QQ\langle\Xcal\rangle               
  \end{tikzcd}
\end{equation}

Clavier \cite{DSRforAMZV} introduced a natural map $\sfk^N : \Hcal_{BCK}^\Ycal \rightarrow \Hcal_{BCK}^\Xcal$ which respects tree structures. However, this map does not make the diagram commutative. In fact, he proved that under his natural map $\sfk^N$ the arborified multiple zeta values of the second kind are less than or equal to the arborified multiple zeta values of the first kind (cf. Theorem \ref{thm Clavier}).

In this paper, building on the work of Foissy \cite{arborification1,arborification2,NBCKHopfalg}, we generalize Manchon's question and Clavier's map to the case of planar rooted trees. We construct the linear map $\phi$ and define $\sfk^{PT}$ as the composition of $\sfk^{PN}$ and $\phi$ and show the following theorem.
\begin{thm}[= Theorem \ref{thm Main theorem}]
  The following diagram is commutative.
  $$\begin{tikzcd}
   & \Hcal_{NBCK}^{P\Ycal} \arrow[rd, "\sfk^{PN}"] & \\
  \Hcal_{NBCK}^{P\Ycal} \arrow[rr, "\sfk^{PT}"] \arrow[ru, "\phi"] \arrow[d, "\afk_{P\Ycal}"'] & & \Hcal_{NBCK}^{P\Xcal} \arrow[d, "\afk_{P\Xcal}"] \\
  \QQ\langle\Ycal\rangle \arrow[rr, "\sfk"]               &  & \QQ\langle\Xcal\rangle               
  \end{tikzcd}$$
\end{thm}
The maps $\afk_{P\Ycal}$, $\sfk^{PN}$, $\afk_{P\Xcal}$ above are the lifts of the maps $\afk_{\Ycal}$, $\sfk^{N}$, $\afk_{\Xcal}$ respectively. Applying this theorem, we can construct a map $\sfk^T$ which makes the diagram (\ref{diagram}) commutative.

In this paper, we review some definitions of rooted trees in \S\ref{section Notations}. In \S\ref{section Arborified multiple zeta values}, we recall two kinds of arborified multiple zeta values and formulate Manchon's question. We prove the above theorem  and present the solution to Manchon’s question in \S\ref{section Main theorem}.

\section{Notations}\label{section Notations}
Arborified multiple zeta values are one of the generalizations of multiple zeta values, which are multiple zeta values associated with rooted trees. In this section, we review the necessary background on rooted trees and define $\Dcal$-decorated rooted trees.

We begin with the definition of a rooted tree.
\begin{defn}\label{defn rooted tree}
  A {\it (non-planar) rooted tree} $T=(T, r)$ is a tree $T = (V(T), E(T))$ in which one vertex $r$ is designated as the {\it root} of the tree $T$, where $V(T)$ denotes the {\it vertex set} of $T$ and $E(T)$ denotes the {\it edge set} of $T$.
\end{defn}

\begin{defn}\label{defn depth function}
  The {\it depth function} $\rho_T$ is a function from $V(T)$ to $\ZZ_{\geq0}$ that maps a vertex $v$ to the length of the path from $r$ to $v$.
\end{defn}

We then recall the definition of a planar rooted tree.
\begin{defn}\label{defn planar rooted tree}
  A {\it planar rooted tree} $T=(T, r, \alpha_T)$ is defined as a rooted tree $(T, r)$ and a {\it total order relation} $\alpha_T \subset V(T) \times V(T)$ on the vertex set $V(T)$ which satisfies
  \begin{enumerate}[(i)]
    \item $\forall u, v\in V(T)$ $\rho_T(u) < \rho_T(v) \Rightarrow (u, v)\in \alpha_T$,
    \item If $\{u, v\}, \{x, y\}$ in $E(T)$, $\rho_T(u)=\rho_T(x)=\rho_T(v)-1=\rho_T(y)-1$ and $(u, x)\in \alpha_T$, then $(v, y)\in \alpha_T$.
  \end{enumerate}
\end{defn}

In this paper, we adopt the opposite tree order because our definition of multiple zeta values follows a different convention from that of  Manchon \cite{AMZV}.
\begin{defn}\label{defn opposite tree order}
  The {\it opposite tree order} of a rooted tree $T$ is a partial order relation $\preceq_T \subset V(T) \times V(T)$ on $V(T)$ defined by $(u, v)\in \preceq_T$ if and only if the path from $r$ to $u$ contains the path from $r$ to $v$. Similarly, for every planar rooted tree $T$, we define the opposite tree order of $T$ in the same way.
\end{defn}

\begin{defn}\label{defn leaf}
  Let $T$ be a rooted tree. A vertex $v$ in $V(T)$ is called a {\it leaf} if $v$ is minimal under the opposite tree order. {\it The set of leaves} of $T$ is denoted by $\text{leaf}(T)$. Similarly, for every planar rooted tree $T$, we define the leaf and the set of leaves of $T$ in the same way.
\end{defn}

The following definition will be needed in Definition \ref{defn error term}.
\begin{defn}\label{defn minimal incomparable pair}
  Let $T$ be a planar rooted tree. Since the vertex set $V(T)$, equipped with the total order relation $\alpha_T$, is totally ordered, the lexicographically ordered set $V(T)\times V(T)\setminus \preceq_T$ is also totally ordered. The pair $(a, b)$ is called the {\it minimal incomparable pair} of $T$ if $(a, b)$ is minimal element of $V(T)\times V(T)\setminus \preceq_T$ with respect to the lexicographic order.
\end{defn}

We adopt the following definition of a forest, which differs from the standard one but is more convenient for our purposes.
\begin{defn}\label{defn forest}
  A {\it forest} $F$ of rooted trees (resp. planar rooted trees) is constructed by
  \begin{enumerate}[step 1:]
    \item Removing the root $r$ from a rooted tree (resp. planar rooted tree) $T$.
    \item For each vertex $v$ with $\rho_T(v) = 1$, designate $v$ as the root of its connected component in the graph $T\setminus \{r\}$.
  \end{enumerate}
  Note that a forest of planar rooted trees inherits a total order relation.
\end{defn}

\begin{rem}\label{rem forest}
  A forest $F$ of rooted trees is a disjoint union of rooted trees. A forest $F$ of planar rooted trees is a disjoint union of rooted trees with a total order relation $\alpha \subset V(F) \times V(F)$. We denote a forest $F$ of rooted trees by $T_1\cdots T_n$, and let $r_i$ be the root of $T_i$. In the case of planar rooted trees, we require that $(r_i, r_j)\in \alpha$ when $i\leq j$.
\end{rem}

\begin{defn}[\cite{CK}, Equation (44), (45), (46)]\label{defn operator}
  Let $\Tcal$ be the set of all rooted trees $T$ (resp. planar rooted tree), and let $\Fcal$ be the set of all rooted forests $F$ (resp. planar rooted forest).
  \begin{enumerate}
    \item  The {\it pruning operator} $B_- : \Tcal \to \Fcal$ is a map that sends a rooted tree $T$ (resp. planar rooted tree) to a rooted forest $F = T_1\cdots T_n$ (resp. planar rooted forest) by removing the root $r$ from the tree:
    $$B_-(T)=T_1\cdots T_n.$$
    \item The {\it grafting operator} $B_+ : \Fcal \to \Tcal$ is a map that sends any rooted forest $F = T_1\cdots T_n$ (resp. planar rooted forest) to a rooted tree $T$ (resp. planar rooted tree) by grafting all components onto the common root:
    $$B_+(T_1\cdots T_n)=T.$$
  \end{enumerate}
  Note that the pruning operator $B_-$ and the grafting operator $B_+$ satisfy the following:
  $$B_-(B_+(T_1\cdots T_n)) = T_1\cdots T_n \text{ and } B_+(B_-(T)) = T.$$
\end{defn}

\begin{defn}\label{defn ladder tree}
  A {\it ladder tree} $T$ (resp. ladder forest $F$) is a tree $T$ (resp. forest $F$) which has no branching vertex (i.e. it has no vertex with degree greater than 2). Similarly, for every planar rooted tree $T$, the ladder planar tree $T$ (resp. ladder planar forest $F$) is defined in the same way.
\end{defn}

We now define a decorated rooted tree, which is the main object of study in this paper.
\begin{defn}\label{defn decorated rooted tree}
  Let $\Dcal$ be a set. A {\it $\Dcal$-decorated rooted tree (resp. forest)} is a rooted tree $D$ (resp. forest $F_D$) equipped with a decoration map
  $$\delta_D : V(D) \to \Dcal \text{ (resp. } \delta_{F_D} : V(F_D) \to V(F_D).$$
  Similarly, for every set $\Dcal$, we define the $\Dcal$-decorated planar rooted tree (resp. forest) in the same way.
\end{defn}

\begin{rem}
  In this paper, we study the case where the set $\Dcal$ is either $\Xcal$ or $\Ycal$, where  $\Xcal \coloneqq \{x_0, x_1\}$ and $\Ycal \coloneqq \{y_n\,|\,n\in \NN\}$.
\end{rem}

\section{Arborified multiple zeta values}\label{section Arborified multiple zeta values}
We begin by recalling two kinds of arborified multiple zeta values and the corresponding Hopf algebras. Next, we introduce morphisms between the relevant spaces. Based on these definitions, we present a question posed by Manchon, and finally review Clavier’s answer to it.
\begin{defn}\label{defn arborified multiple zeta values of the first kind}
  {\it Arborified multiple zeta values of the first kind} are multiple zeta values associated with a $\Ycal$-decorated rooted tree $Y$ (resp. forest), defined as the harmonic series associated to the triple $(V(Y), \preceq_Y, \delta_Y)$.
  $$\zeta(Y) \coloneqq \sum_{\substack{n_v\in \NN\\n_u<n_v\ \text{if}\ u\prec_Y v}} \prod_{v\in V(Y)} \frac{1}{n_v^{k_v}},$$
  where $k_v$ is the integer $n$ such that $\delta_Y(v) = y_n$.
  The series converges when the root is decorated by $y_n$ for some $n\geq 2$. Similarly, for a $\Ycal$-decorated planar rooted tree (resp. forest), we define the arborified multiple zeta values of the first kind in the same way. The corresponding series also converges when the root is decorated by $y_n$ with $n\geq 2$.
\end{defn}

\begin{defn}\label{defn arborified multiple zeta values of the second kind}
  {\it Arborified multiple zeta values of the second kind} are multiple zeta values associated with a $\Xcal$-decorated rooted tree $X$ (resp. forest), defined as Yamamoto's integral associated with the triple $(V(X), \preceq_X, \delta_X)$.
  $$\zeta(X) \coloneqq I(X) = \int_{\Delta(X)} \prod_{v\in V(X)}\omega_{\delta_X(v)}(t_v),$$
  where
  $$\Delta(X) \coloneqq \{\mathbf{t} = (t_v)_{v\in V(X)} \in (0,1)^{V(X)} \,|\, t_u < t_v \text{ if } u\prec_X v\},$$
  $$\omega_{x_0}(t) \coloneqq \frac{dt}{t},\ \omega_{x_1}(t) \coloneqq \frac{dt}{1-t}.$$
  The integral converges when the root is decorated by $x_0$ and the leaves are decorated by $x_1$. Similarly, for a $\Xcal$-decorated planar rooted tree (resp. forest), we define the arborified multiple zeta values of the second kind in the same way. The corresponding integral also converges when the root is decorated by $x_0$ and the leaves are decorated by $x_1$.
\end{defn}

The following Hopf algebra structures and algebra morphisms play an important role in the theory of MZVs.

\begin{defn}[\cite{Hopfalg}]\label{defn Hopf algebra}
  \begin{enumerate}[(i)]
    \item Let $\QQ\langle\Xcal\rangle$ be the non-commutative polynomial algebra generated by $\Xcal$.\\
    The triple $(\QQ\langle\Xcal\rangle, \shuffle, \Delta)$ is a graded commutative Hopf algebra, where $\shuffle$ denotes the shuffle product, and $\Delta$ is the coproduct.
    \item Let $\QQ\langle\Ycal\rangle$ be the non-commutative polynomial algebra generated by $\Ycal$.\\
    The triple $(\QQ\langle\Ycal\rangle, \ast, \Delta)$ is a graded commutative Hopf algebra, where $\ast$ denotes the stuffle product, and $\Delta$ is the coproduct.
  \end{enumerate}
\end{defn}

\begin{defn}\label{defn map s}
  The algebra morphism $\sfk : \QQ\langle\Ycal\rangle \rightarrow \QQ\langle\Xcal\rangle$ is defined by
  $$\sfk(y_{n_1}\cdots y_{n_d}) = x_1x_0^{n_1-1}\cdots x_1x_0^{n_d-1}.$$
\end{defn}

Manchon introduced the following Hopf algebra structure on $\Dcal$-decorated rooted trees.

\begin{defn}[\cite{CKHopfalg}, \S 5 and \cite{AMZV},\S 3]\label{defn HBCKD}
  Let $\Dcal$ be a set. Let $\QQ[\Tcal^\Dcal]$ be the commutative polynomial algebra, where $\Tcal^\Dcal$ is the set of non-empty $\Dcal$-decorated rooted trees. {\it Butcher-Connes-Kreimer Hopf algebra} (BCK Hopf algebra) of $\Dcal$-decorated rooted trees $\Hcal_{BCK}^{\Dcal}$ is defined as the triple $(\QQ[\Tcal^\Dcal], \pi, \Delta)$ which is a graded non-commutative Hopf algebra with the product $\pi$ and the coproduct $\Delta$ (see \cite{AMZV}, Equation (18)).
\end{defn}

In this paper, we consider the case of $\Dcal$-decorated planar rooted tree.

\begin{defn}[\cite{arborification1}, \S 5 and \cite{NBCKHopfalg}, \S 1]\label{defn kTPD}
  Let $\Dcal$ be a set. Let $\QQ\langle \Tcal^{P\Dcal} \rangle$ be the non-commutative polynomial algebra, where $\Tcal^{P\Dcal}$ is the set of non-empty $\Dcal$-decorated planar rooted trees. {\it non-commutative Butcher-Connes-Kreimer Hopf algebra} (NBCK Hopf algebra) of $\Dcal$-decorated planar rooted trees $\Hcal_{NBCK}^{P\Dcal}$ is defined as the triple $(\QQ\langle \Tcal^{P\Dcal} \rangle, \pi, \Delta)$ which is a graded non-commutative Hopf algebra with the product $\pi$ and the coproduct $\Delta$ (see \cite{arborification1}, Theorem 29).
\end{defn}

The only distinction between a planar rooted tree and a rooted tree lies in the presence of a total order relation $\alpha$. Hence, we have the following natural projection:

\begin{defn}\label{defn projection}
  The natural projection from $\QQ\langle \Tcal^{P\Dcal} \rangle$ to $\QQ[\Tcal^\Dcal]$ by removing the total order relation is denoted by $\widehat{\alpha}_\Dcal$, which is an algebra morphism.
\end{defn}

We give the following definition of grafting operator for the $\Dcal$-decorated rooted trees and the $\Dcal$-decorated planar rooted trees.

\begin{defn}[\cite{arborification2}, \S 1]\label{defn grafting operator}
  Let $\Dcal$ be a set, and $d$ be an element in $\Dcal$. The {\it grafting operator} $B_+^d : \Hcal_{BCK}^{\Dcal} \to \Hcal_{BCK}^{\Dcal}$ (resp. $B_+^d : \Hcal_{NBCK}^{P\Dcal} \to \Hcal_{NBCK}^{P\Dcal}$) is an algebra morphism that maps any $\Dcal$-decorated rooted forest (resp. $\Dcal$-decorated planar rooted forest) to a $\Dcal$-decorated rooted tree (resp. $\Dcal$-decorated planar rooted tree) by grafting all components onto the common root decorated by $d$.
\end{defn}

We use the same notation $B_-$ for the pruning operator of the $\Dcal$-decorated rooted trees and the $\Dcal$-decorated planar rooted trees. Using the notation of the minimal incomparable pair (cf. Definition \ref{defn minimal incomparable pair}) for a planar rooted tree, we obtain the following lemma.

\begin{lem}\label{lem expression of planar rooted tree}
  Let $D$ be a $\Dcal$-decorated planar rooted tree. Then $D$ has the following expression:
  $$D = B_+^{d_1} \circ \cdots \circ B_+^{d_s}(B_+^{d_a}(F_a)\,B_+^{d_b}(F_b)\,F)$$
  for some $\Dcal$-decorated planar forest $F, F_a, F_b$, where $(a, b)$ is the minimal incomparable pair of $D$ and $d_{a} = \delta_D(a), d_{b} = \delta_D(b)$.
\end{lem}

\begin{proof}
  Let $D = (V(D), E(D), r, \alpha_D)$ be a $\Dcal$-decorated planar rooted tree, with $(a, b)$ its minimal incomparable pair. By Definition \ref{defn minimal incomparable pair}, the opposite tree order $\preceq_D$, when restricted to
  $$\{v\in V(D) | (v, a)\in \alpha_D\}^2$$
  induces a total order. Hence using the notation in Remark \ref{rem forest}, we obtain the following expression:
  $$D = B_+^{d_1} \circ \cdots \circ B_+^{d_s}(B_+^{d_a}(F_a)\,B_+^{d_b}(F_b)\,F),$$
  which is this lemma.
\end{proof}

\begin{thm}[\cite{arborification1}, Theorem 31 and \cite{AMZV}, Equation (22)]\label{thm arborification}
  Let $\Hcal$ be a graded Hopf algebra, $d$ be an element in $\Dcal$. For any Hochschild one-cocycle $L^d : \Hcal \rightarrow \Hcal$, that is, a linear map satisfying
  \begin{equation}\label{eq one-cocycle condition}
    \Delta(L^d(x)) = L^d(x)\otimes \mathbf{1}_\Hcal + (Id\otimes L^d) \circ \Delta(x).
  \end{equation}
  Then there exists a unique Hopf algebra morphism $\Phi_L : \Hcal_{BCK}^\Dcal \rightarrow \Hcal$ satisfying
  $$\Phi_L \circ B_+^d = L^d \circ \Phi_L.$$
\end{thm}

\begin{eg}[\cite{AMZV}, \S 4]\label{defn Hochschild one-cocycle}
  Let $\Dcal$ be $\Xcal$ or $\Ycal$ and $d$ be an element in $\Dcal$. Then, the linear map $L^{d} : \QQ\langle\Dcal\rangle \rightarrow \QQ\langle\Dcal\rangle$, which is defined by
  $$L^{d}(d_1\cdots d_s) \coloneqq d_1\cdots d_s d,$$
  is a Hochschild one-cocycle on $\QQ\langle\Dcal\rangle$.
  The one-cocycle condition \eqref{eq one-cocycle condition} for $L^{d}$ follows directly from the structure of the deconcatenation coproduct.
\end{eg}

By the theorem above, Manchon introduced the simple arborification and the contracting arborification as follows:

\begin{defn}[\cite{AMZV}, \S 4]\label{defn arborification}
  \begin{enumerate}[(i)]
    \item The {\it simple arborification} is the unique Hopf algebra morphism
    $$\afk_\Xcal : \Hcal_{BCK}^\Xcal \rightarrow \QQ\langle\Xcal\rangle$$
    satisfying $\afk_\Xcal \circ B_+^{x_m} = L^{x_m} \circ \afk_\Xcal$, where $m = 0$ or $1$.
    \item The {\it contracting arborification} is the unique Hopf algebra morphism
    $$\afk_\Ycal : \Hcal_{BCK}^\Ycal \rightarrow \QQ\langle\Ycal\rangle$$
    satisfying $\afk_\Ycal \circ B_+^{y_n} = L^{y_n} \circ \afk_\Ycal$, where $n \in \ZZ_{>0}$.
  \end{enumerate}
\end{defn}

In the case of planar rooted trees, we introduce the lift of the simple arborification and the contracting arborification as follows:

\begin{defn}\label{defn lift arborification}
   The lifting map $\afk_{P\Xcal} : \Hcal_{NBCK}^{P\Xcal} \rightarrow \QQ\langle\Xcal\rangle$ and $\afk_{P\Ycal} : \Hcal_{NBCK}^{P\Ycal} \rightarrow \QQ\langle\Ycal\rangle$ are defined by $\widehat{\alpha}_\Xcal \circ \afk_\Xcal$ and $\widehat{\alpha}_\Ycal \circ \afk_\Ycal$ (for $\widehat{\alpha}_\Xcal$ and $\widehat{\alpha}_\Ycal$ see Definition \ref{defn projection}), respectively. They are algebra morphisms, called the lifting maps of $\afk_\Xcal$ and $\afk_\Ycal$, respectively.
\end{defn}

The following lemma will be used in the proof of Lemma \ref{lem arborification}.

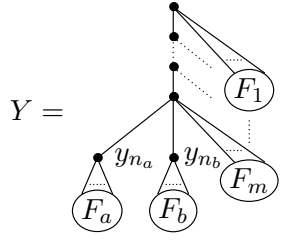
\begin{lem}\label{lem lambda-shuffle product}
  Let $F$ and $F'$ be the $\Xcal$-decorated rooted forests. Then, the simple arborification $\afk_\Xcal$ satisfies the following equation:
  $$\afk_\Xcal (B_+^{x}(F) \, B_+^{x'}(F')) = L^{x} \circ \afk_\Xcal(B_+^{x'}(F') \, F)) + L^{x'} \circ \afk_\Xcal(B_+^{x}(F) \, F')).$$
  Let $F$ and $F'$ be the $\Ycal$-decorated rooted forest. Then, the contracting arborification $\afk_\Ycal$ satisfies the following equation:
  \begin{align*}
    \afk_\Ycal (B_+^{y_n}(F) \, B_+^{y_{n'}}(F')) =  & L^{y_n} \circ \afk_\Ycal(B_+^{y_{n'}}(F') \, F)) + L^{y_{n'}} \circ \afk_\Ycal(B_+^{y_n}(F) \, F')) \\
      &  + L^{y_{n+n'}} \circ \afk_\Ycal(F \, F')).
  \end{align*}
  The same statement also holds for $\afk_{P\Xcal}$ and $\afk_{P\Ycal}$, i.e., for the case of planar rooted forests.
\end{lem}

\begin{proof}
  Let $\Dcal$ be $\Xcal$ or $\Ycal$. By Definition \ref{defn lift arborification}, the equality $\afk_{P\Dcal} = \hat{\alpha}_\Dcal \circ \afk_\Dcal$ implies that if the lemma holds for $\afk_{P\Dcal}$ then it holds for $\afk_\Dcal$. On the other hand, if the lemma holds for $\afk_\Dcal$, then it holds independently of total order relation. Hence, the lemma holds for $\afk_{P\Dcal}$. Thus, it suffices to prove the case $\afk_\Dcal$, and the proofs for $\afk_\Xcal$ and $\afk_\Ycal$ are similar; we will therefore prove only $\afk_\Xcal$.\\
  \begin{align*}
    \afk_\Xcal (B_+^{x}(F) \, B_+^{x'}(F')) = & \afk_\Xcal (B_+^{x}(F)) \shuffle \afk_\Xcal (B_+^{x'}(F')) \\
    = & L^{x} \circ \afk_\Xcal(F) \shuffle L^{x'} \circ \afk_\Xcal(F') \\
    = & L^{x}(\afk_\Xcal(F) \shuffle L^{x'} \circ \afk_\Xcal(F')) + L^{x'}(L^{x} \circ \afk_\Xcal(F) \shuffle \afk_\Xcal(F')) \\
    = & L^{x}(\afk_\Xcal(F) \shuffle \afk_\Xcal (B_+^{x'}(F'))) + L^{x'}(\afk_\Xcal (B_+^{x}(F)) \shuffle \afk_\Xcal(F')) \\
    = & L^{x} \circ \afk_\Xcal(B_+^{x'}(F') \, F)) + L^{x'} \circ \afk_\Xcal(B_+^{x}(F) \, F'))
  \end{align*}
\end{proof}

The following lemma describes an important structural property of arborifications.

\begin{lem}\label{lem arborification}
  Let
  $$X = \begin{xy}
    {(0,14) \ar @{{*}-{*}} (0,10)},
    {(0,10) \ar @{{*}.{*}} (0,6)},
    {(0,6) \ar @{{*}-{*}} (0,2)},
    {(0,2) \ar @{{*}-{*}} (0,-6)},
    {(0,2) \ar @{{*}-{}} (8,-6.5)},
    {(0,2) \ar @{{*}-{}} (12,-6.5)},
    {(0,6) \ar @{{*}.{}} (5,2)},
    {(0,10) \ar @{{*}.{}} (5,6)},
    {(0,14) \ar @{{*}-{}} (8,5.3)},
    {(0,14) \ar @{{*}-{}} (12,5.3)},
    {(0,2) \ar @{{*}-{*}} (-10,-6)},
    {(10,-1) \ar @{{}.{}} (10,-4)},
    {(7,-5) \ar @{{}.{}} (9,-5)},
    {(7,7) \ar @{{}.{}} (9,7)},
    {(0,-6) \ar @{{*}-{}} (-2,-10.7)}, (4,-6)*{x_b},
    {(0,-6) \ar @{{*}-{}} (2,-10.7)},
    {(-1,-9.5) \ar @{{}.{}} (1,-9.5)},
    {(-10,-6) \ar @{{*}-{}} (-8,-10.7)}, (-6,-6)*{x_a},
    {(-10,-6) \ar @{{*}-{}} (-12,-10.7)},
    {(-9,-9.5) \ar @{{}.{}} (-11,-9.5)},
    (-10,-13)*+[o][F-]{F_a},
    (0,-13)*+[o][F-]{F_b},
    (10,-9)*+[o][F-]{F_m},
    (10,3)*+[o][F-]{F_1},
  \end{xy}$$
  be a $\Xcal$-decorated rooted tree, where $F_1, \ldots, F_m, F_a, F_b$ are $\Xcal$-decorated rooted forests. Then, the simple arborification $\afk_\Xcal$ satisfies the following equation:
  $$\afk_\Xcal (X) = \afk_\Xcal \left( \begin{xy}
    {(0,16) \ar @{{*}-{*}} (0,12)},
    {(0,12) \ar @{{*}.{*}} (0,8)},
    {(0,8) \ar @{{*}-{*}} (0,4)},
    {(0,4) \ar @{{*}-{*}} (0,-4)},
    {(0,4) \ar @{{*}-{}} (10,-4)},
    {(0,8) \ar @{{*}.{}} (5,4)},
    {(0,12) \ar @{{*}.{}} (5,8)},
    {(0,16) \ar @{{*}-{}} (10,8)},
    {(0,-8) \ar @{{*}-{}} (-2,-12.7)},
    {(1,-11.5) \ar @{{}.{}} (-1,-11.5)},
    {(0,-8) \ar @{{*}-{}} (2,-12.7)},
    {(10,-3) \ar @{{}.{}} (10,1)},
    {(0,-4) \ar @{{*}-{}} (0,-8)}, (-4,-4)*{x_b},
    {(0,-4) \ar @{{*}-{}} (10,-12)}, (-4,-8)*{x_a},
    {(0,16) \ar @{{*}-{}} (7,6)},
    {(0,4) \ar @{{*}-{}} (7,-5)},
    {(0,-4) \ar @{{*}-{}} (7,-14)},
    {(6,9) \ar @{{}.{}} (8,9)},
    {(6,-3) \ar @{{}.{}} (8,-3)},
    {(6,-11) \ar @{{}.{}} (8,-11)},
    (0,-15)*+[o][F-]{F_a},
    (10,-15)*+[o][F-]{F_b},
    (10,-7)*+[o][F-]{F_m},
    (10,5)*+[o][F-]{F_1},
  \end{xy} \right) + \afk_\Xcal \left( \begin{xy}
    {(0,16) \ar @{{*}-{*}} (0,12)},
    {(0,12) \ar @{{*}.{*}} (0,8)},
    {(0,8) \ar @{{*}-{*}} (0,4)},
    {(0,4) \ar @{{*}-{*}} (0,-4)},
    {(0,4) \ar @{{*}-{}} (10,-4)},
    {(0,8) \ar @{{*}.{}} (5,4)},
    {(0,12) \ar @{{*}.{}} (5,8)},
    {(0,16) \ar @{{*}-{}} (10,8)},
    {(0,-8) \ar @{{*}-{}} (-2,-12.7)},
    {(1,-11.5) \ar @{{}.{}} (-1,-11.5)},
    {(0,-8) \ar @{{*}-{}} (2,-12.7)},
    {(10,-3) \ar @{{}.{}} (10,1)},
    {(0,-4) \ar @{{*}-{}} (0,-8)}, (-4,-4)*{x_a},
    {(0,-4) \ar @{{*}-{}} (10,-12)}, (-4,-8)*{x_b},
    {(0,16) \ar @{{*}-{}} (7,6)},
    {(0,4) \ar @{{*}-{}} (7,-5)},
    {(0,-4) \ar @{{*}-{}} (7,-14)},
    {(6,9) \ar @{{}.{}} (8,9)},
    {(6,-3) \ar @{{}.{}} (8,-3)},
    {(6,-11) \ar @{{}.{}} (8,-11)},
    (0,-15)*+[o][F-]{F_b},
    (10,-15)*+[o][F-]{F_a},
    (10,-7)*+[o][F-]{F_m},
    (10,5)*+[o][F-]{F_1},
  \end{xy} \right).$$
  
  Let
  $$Y = \begin{xy}
    {(0,14) \ar @{{*}-{*}} (0,10)},
    {(0,10) \ar @{{*}.{*}} (0,6)},
    {(0,6) \ar @{{*}-{*}} (0,2)},
    {(0,2) \ar @{{*}-{*}} (0,-6)},
    {(0,2) \ar @{{*}-{}} (8,-6.5)},
    {(0,2) \ar @{{*}-{}} (12,-6.5)},
    {(0,6) \ar @{{*}.{}} (5,2)},
    {(0,10) \ar @{{*}.{}} (5,6)},
    {(0,14) \ar @{{*}-{}} (8,5.3)},
    {(0,14) \ar @{{*}-{}} (12,5.3)},
    {(0,2) \ar @{{*}-{*}} (-10,-6)},
    {(10,-1) \ar @{{}.{}} (10,-4)},
    {(7,-5) \ar @{{}.{}} (9,-5)},
    {(7,7) \ar @{{}.{}} (9,7)},
    {(0,-6) \ar @{{*}-{}} (-2,-10.7)}, (4,-6)*{y_{n_b}},
    {(0,-6) \ar @{{*}-{}} (2,-10.7)},
    {(-1,-9.5) \ar @{{}.{}} (1,-9.5)},
    {(-10,-6) \ar @{{*}-{}} (-8,-10.7)}, (-5,-6)*{y_{n_a}},
    {(-10,-6) \ar @{{*}-{}} (-12,-10.7)},
    {(-9,-9.5) \ar @{{}.{}} (-11,-9.5)},
    (-10,-13)*+[o][F-]{F_a},
    (0,-13)*+[o][F-]{F_b},
    (10,-9)*+[o][F-]{F_m},
    (10,3)*+[o][F-]{F_1},
  \end{xy}$$
  be a $\Ycal$-decorated rooted tree, where $F_1, \ldots, F_m, F_a, F_b$ are $\Ycal$-decorated rooted forests. Then, the contracting arborification $\afk_\Ycal$ satisfies the following equation:
  $$
     \afk_\Ycal (Y) = \afk_\Ycal \left( \begin{xy}
    {(0,16) \ar @{{*}-{*}} (0,12)},
    {(0,12) \ar @{{*}.{*}} (0,8)},
    {(0,8) \ar @{{*}-{*}} (0,4)},
    {(0,4) \ar @{{*}-{*}} (0,-4)},
    {(0,4) \ar @{{*}-{}} (10,-4)},
    {(0,8) \ar @{{*}.{}} (5,4)},
    {(0,12) \ar @{{*}.{}} (5,8)},
    {(0,16) \ar @{{*}-{}} (10,8)},
    {(0,-8) \ar @{{*}-{}} (-2,-12.7)},
    {(1,-11.5) \ar @{{}.{}} (-1,-11.5)},
    {(0,-8) \ar @{{*}-{}} (2,-12.7)},
    {(10,-3) \ar @{{}.{}} (10,1)},
    {(0,-4) \ar @{{*}-{}} (0,-8)}, (-4,-4)*{y_{n_b}},
    {(0,-4) \ar @{{*}-{}} (10,-12)}, (-4,-8)*{y_{n_a}},
    {(0,16) \ar @{{*}-{}} (7,6)},
    {(0,4) \ar @{{*}-{}} (7,-5)},
    {(0,-4) \ar @{{*}-{}} (7,-14)},
    {(6,9) \ar @{{}.{}} (8,9)},
    {(6,-3) \ar @{{}.{}} (8,-3)},
    {(6,-11) \ar @{{}.{}} (8,-11)},
    (0,-15)*+[o][F-]{F_a},
    (10,-15)*+[o][F-]{F_b},
    (10,-7)*+[o][F-]{F_m},
    (10,5)*+[o][F-]{F_1},
  \end{xy} \right) + \afk_\Ycal \left( \begin{xy}
    {(0,16) \ar @{{*}-{*}} (0,12)},
    {(0,12) \ar @{{*}.{*}} (0,8)},
    {(0,8) \ar @{{*}-{*}} (0,4)},
    {(0,4) \ar @{{*}-{*}} (0,-4)},
    {(0,4) \ar @{{*}-{}} (10,-4)},
    {(0,8) \ar @{{*}.{}} (5,4)},
    {(0,12) \ar @{{*}.{}} (5,8)},
    {(0,16) \ar @{{*}-{}} (10,8)},
    {(0,-8) \ar @{{*}-{}} (-2,-12.7)},
    {(1,-11.5) \ar @{{}.{}} (-1,-11.5)},
    {(0,-8) \ar @{{*}-{}} (2,-12.7)},
    {(10,-3) \ar @{{}.{}} (10,1)},
    {(0,-4) \ar @{{*}-{}} (0,-8)}, (-4,-4)*{y_{n_a}},
    {(0,-4) \ar @{{*}-{}} (10,-12)}, (-4,-8)*{y_{n_b}},
    {(0,16) \ar @{{*}-{}} (7,6)},
    {(0,4) \ar @{{*}-{}} (7,-5)},
    {(0,-4) \ar @{{*}-{}} (7,-14)},
    {(6,9) \ar @{{}.{}} (8,9)},
    {(6,-3) \ar @{{}.{}} (8,-3)},
    {(6,-11) \ar @{{}.{}} (8,-11)},
    (0,-15)*+[o][F-]{F_b},
    (10,-15)*+[o][F-]{F_a},
    (10,-7)*+[o][F-]{F_m},
    (10,5)*+[o][F-]{F_1},
  \end{xy} \right) + \afk_\Ycal \left( \begin{xy}
    {(0,16) \ar @{{*}-{*}} (0,12)},
    {(0,12) \ar @{{*}.{*}} (0,8)},
    {(0,8) \ar @{{*}-{*}} (0,4)},
    {(0,4) \ar @{{*}-{*}} (0,-4)},
    {(0,4) \ar @{{*}-{}} (10,-4)},
    {(0,8) \ar @{{*}.{}} (5,4)},
    {(0,12) \ar @{{*}.{}} (5,8)},
    {(0,16) \ar @{{*}-{}} (10,8)},
    {(10,-3) \ar @{{}.{}} (10,1)}, (-6,-4)*{y_{n_{a}+n_{b}}},
    {(0,-4) \ar @{{*}-{}} (-2,-12.7)},
    {(1,-11) \ar @{{}.{}} (-1,-11)},
    {(0,-4) \ar @{{*}-{}} (2,-12.7)},
    {(0,-4) \ar @{{*}-{}} (10,-12)},
    {(0,16) \ar @{{*}-{}} (7,6)},
    {(0,4) \ar @{{*}-{}} (7,-5)},
    {(0,-4) \ar @{{*}-{}} (7,-14)},
    {(6,9) \ar @{{}.{}} (8,9)},
    {(6,-3) \ar @{{}.{}} (8,-3)},
    {(6,-11) \ar @{{}.{}} (8,-11)},
    (0,-15)*+[o][F-]{F_a},
    (10,-15)*+[o][F-]{F_b},
    (10,-7)*+[o][F-]{F_m},
    (10,5)*+[o][F-]{F_1},
  \end{xy} \right).$$
  
  The same statement also holds for $\afk_{P\Xcal}$ and $\afk_{P\Ycal}$, i.e., for the case of planar rooted forests.
\end{lem}

\begin{proof}
  Let $\Dcal$ be $\Xcal$ or $\Ycal$. By Definition \ref{defn lift arborification}, the equality $\afk_{P\Dcal} = \hat{\alpha}_\Dcal \circ \afk_\Dcal$ implies that if the lemma holds for $\afk_{P\Dcal}$ then it holds for $\afk_\Dcal$. On the other hand, if the lemma holds for $\afk_\Dcal$, then it holds independently of total order relation. Hence, the lemma holds for $\afk_{P\Dcal}$. Thus, it suffices to prove the case $\afk_\Dcal$, and the proofs for $\afk_\Xcal$ and $\afk_\Ycal$ are similar; we will therefore prove only $\afk_\Xcal$.\\
  By Definition \ref{defn arborification}, $\afk_\Xcal$ satisfies the following equation:
  \begin{equation}\label{eq arborification}
    \afk_\Xcal \circ B_+^{x_m} = L^{x_m} \circ \afk_\Xcal.
  \end{equation}
  Let $F_{i}$ be an $\Xcal$-decorated rooted forest. Using the equation \eqref{eq arborification}, we obtain that when
  $$\afk_\Xcal(F_{1}) = \afk_\Xcal(F_{2}),$$
  that we have
  \begin{align*}
    \afk_\Xcal \circ B_+^{x_m}(F_{1} \, F_{3}) = & L^{x_m} \circ \afk_\Xcal(F_{1} \, F_{3}) \\
    = & L^{x_m}(\afk_\Xcal(F_{1}) \shuffle \afk_\Xcal(F_{3})) \\
    = & L^{x_m}(\afk_\Xcal(F_{2}) \shuffle \afk_\Xcal(F_{3})) \\
    = & L^{x_m} \circ \afk_\Xcal(F_{2} \, F_{3}) \\
    = & \afk_\Xcal \circ B_+^{x_m}(F_{2} \, F_{3}).
  \end{align*}
  Thus, it remains to prove the following equation:
  $$\afk_\Xcal (B_+^{x_a}(F_{4}) \, B_+^{x_b}(F_{5})) = \afk_\Xcal (B_+^{x_a}(B_+^{x_b}(F_{5}) \, F_{4})) + \afk_\Xcal (B_+^{x_b}(B_+^{x_a}(F_{4}) \, F_{5})).$$
  Using Equation \eqref{eq arborification} again, we obtain that the right-hand side is equal to the following:
  $$L^{x_a} \circ \afk_\Xcal(B_+^{x_b}(F_{5}) \, F_{4})) + L^{x_b} \circ \afk_\Xcal(B_+^{x_a}(F_{4}) \, F_{5})),$$
  which is equal to the left-hand side by Lemma \ref{lem lambda-shuffle product}.
\end{proof}

Before we state Manchon’s question, we give the following definition.

\begin{defn}\label{defn ladder tree section}
  The {\it ladder tree section} $\ell_\Xcal$ of the simple arborification $\afk_\Xcal$ (resp. $\ell_\Ycal$ of the contracting arborification $\afk_\Ycal$) is defined by
  $$
  \ell_\Xcal (x_{m_1}x_{m_2}\cdots x_{m_s}) =
  \begin{xy}
    (0,-4)*++!L{\scriptstyle x_{m_1}},
    (0,0)*++!L{\scriptstyle x_{m_2}},
    (0,4)*++!L{\scriptstyle x_{m_s}},
    {(0,-4) \ar @{{*}-{*}} (0,0)},
    {(0,0) \ar @{{*}.{*}} (0,4)}
  \end{xy}
  \ \left( \text{resp.} \ 
  \ell_\Ycal (y_{n_1}y_{n_2}\cdots y_{n_t}) =
  \begin{xy}
    (0,-4)*++!L{\scriptstyle y_{n_1}},
    (0,0)*++!L{\scriptstyle y_{n_2}},
    (0,4)*++!L{\scriptstyle y_{n_t}},
    {(0,-4) \ar @{{*}-{*}} (0,0)},
    {(0,0) \ar @{{*}.{*}} (0,4)}
  \end{xy}\right).$$
  Note that ladder tree has a unique total order relation $\alpha$, therefore $\ell_\Xcal$ (resp. $\ell_\Ycal$) is also a section of $\afk_{P\Xcal}$ (resp. $\afk_{P\Ycal}$). In this case, we use the notation $\ell_{P\Xcal}$ and $\ell_{P\Ycal}$.
\end{defn}

\begin{ques}[\cite{AMZV}]
  Manchon posed a question to find a natural map $\sfk^T$ respecting the tree structures, which makes the diagram \eqref{diagram} commutative.
\end{ques}
An obvious answer was given by Manchon \cite{AMZV}, namely:
$$\sfk^T = \ell_\Xcal \circ \sfk \circ \afk_\Ycal.$$
While it renders the diagram commutative, it has the drawback of completely destroying the geometry of trees. There is an attempt by Clavier \cite{DSRforAMZV}.

\begin{defn}\label{defn map sN}
  Let $Y = B_+^{y_n}(Y_1\cdots Y_m)$ be a $\Ycal$-decorated rooted tree in $\Hcal_{BCK}^\Ycal$. The linear map $\sfk^N : \Hcal_{BCK}^\Ycal \rightarrow \Hcal_{BCK}^\Xcal$ is defined recursively by
  $$\sfk^{N}(B_+^{y_n}(Y_1\cdots Y_m)) = (B_+^{x_0})^n\circ B_+^{x_1}(\sfk^{N}(Y_1\cdots Y_m)),$$
  where
  $$\sfk^{N}(Y_1\cdots Y_m) = \sfk^{N}(Y_1) \cdots \sfk^{N}(Y_m),$$
  which is a forest of $\Xcal$-decorated rooted trees.
\end{defn}

In the case of planar rooted trees, we consider the lift of $\sfk^N$ as follows:

\begin{defn}\label{defn map sPN}
  Let $Y = B_+^{y_n}(Y_1\cdots Y_m)$ be a $\Ycal$-decorated planar rooted tree in $\Hcal_{NBCK}^{P\Ycal}$. The linear map $\sfk^{PN} : \Hcal_{NBCK}^{P\Ycal} \rightarrow \Hcal_{NBCK}^{P\Xcal}$ is defined recursively by
  $$\sfk^{PN}(B_+^{y_n}(Y_1\cdots Y_m)) = (B_+^{x_0})^n\circ B_+^{x_1}(\sfk^{PN}(Y_1\cdots Y_m)),$$
  where
  $$\sfk^{PN}(Y_1\cdots Y_m) = \sfk^{PN}(Y_1) \cdots \sfk^{PN}(Y_m),$$
  which is a forest of $\Xcal$-decorated planar rooted trees.
\end{defn}

The following lemma states a trivial but important fact: the map $\sfk^{PN}$, just like $\sfk^N$, also satisfies the commutativity condition for ladder trees.

\begin{lem}\label{lem map sPN}
  The following diagram is commutative.
  $$\begin{tikzcd}
   \Hcal_{NBCK}^{P\Ycal} \arrow[r, "\sfk^{PN}"]  & \Hcal_{NBCK}^{P\Xcal} \arrow[d, "\afk_{P\Xcal}"] \\
  \QQ\langle\Ycal\rangle \arrow[r, "\sfk"] \arrow[u, "\ell_{P\Ycal}"'] & \QQ\langle\Xcal\rangle               
  \end{tikzcd}$$
\end{lem}

\begin{proof}
  This follows from the following direct computation.
  \begin{align*}
    \afk_{P\Xcal} \circ \sfk^{PN} \circ \ell_{P\Ycal}(y_{n_1}\cdots y_{n_m}) =  & \afk_{P\Xcal} \circ \sfk^{PN} (B_+^{y_{n_m}} \circ \cdots \circ B_+^{y_{n_1}}(\emptyset)) \\
    = & \afk_{P\Xcal} ((B_+^{x_{0}})^{n_m} \circ B_+^{x_{1}} (\sfk^{PN} (B_+^{y_{n_{m-1}}} \circ \cdots \circ B_+^{y_{n_1}}(\emptyset)))) \\
    = & \afk_{P\Xcal} ((B_+^{x_{0}})^{n_m} \circ B_+^{x_{1}} \circ \cdots \circ (B_+^{x_{0}})^{n_1} \circ B_+^{x_{1}} (\emptyset)) \\
    = & x_{0}^{n_m} x_{1} \cdots x_{0}^{n_1} x_{1} \\
    = & \sfk(y_{n_1}\cdots y_{n_m})
  \end{align*}
\end{proof}

The following theorem was proved by Clavier \cite{DSRforAMZV}.

\begin{thm}[Clavier \cite{DSRforAMZV}]\label{thm Clavier}
  Let $F$ be a forest in $\Hcal_{BCK}^\Ycal$. If $\zeta(F)$ converges, then we have
  $$\zeta(\sfk^N(F))\leq \zeta(F).$$
  Furthermore, the equality holds if, and only if, $F$ is a ladder forest.
\end{thm}

By definition, we have $\zeta(F) = \zeta(\sfk(\afk_{\Ycal}(F)))$ and $\zeta(\sfk^N(F)) = \zeta(\afk(\sfk^N(F)))$. Thus, by Theorem \ref{thm Clavier}, we see that the diagram \eqref{diagram} fails to commute when $\sfk^T = \sfk^N$.

\section{Main theorem}\label{section Main theorem}
In this section, we consider planar rooted trees and reformulate Manchon’s question in the planar setting. We then construct a linear map that modifies Clavier’s map and show that the lifted diagram becomes commutative with this linear map. This map is then used to construct a solution that satisfies the commutativity condition in Manchon’s original question.

Our goal is to find a solution that makes the diagram commutative while preserving the tree structure as much as possible. We begin by considering the following example.
\begin{eg}\label{eg simplest example}
  Let $Y$ be the $\Ycal$-decorated rooted tree in $\Hcal_{BCK}^\Ycal$ depicted as
  \begin{xy}
    {(-4,-2)*++!R{\scriptstyle y_a} \ar @{{*}-{*}} (0,2)*++!D{\scriptstyle y_b}},
    {(0,2) \ar @{{*}-{*}} (4,-2)*++!L{\scriptstyle y_c}}
  \end{xy}
  . It is the simplest example of a non-ladder forest. From Clavier's theorem (Theorem \ref{thm Clavier}), we know that 
  $$\afk_\Xcal\circ \sfk^N(Y) \neq \sfk \circ \afk_\Ycal(Y).$$
  In order to construct a map $\sfk^T$ which makes the diagram \eqref{diagram} commutative, we consider the error term
  $$\ell_\Xcal(\sfk \circ \afk_\Ycal(Y) - \afk_\Xcal\circ \sfk^N(Y)),$$
  and in this case, the error term can be computed by Lemma \ref{lem arborification} as follows:
  \begin{align*}
    & \ell_\Xcal(\sfk \circ \afk_\Ycal(Y) - \afk_\Xcal\circ \sfk^N(Y)) =  \\
    & \sfk^N\left(
    \begin{xy}
      (0,-2)*++!L{\scriptstyle y_{a+c}},
      (0,2)*++!L{\scriptstyle y_b},
      {(0,-2) \ar @{{*}-{*}} (0,2)}
    \end{xy} - 
    \sum_{i=1}^{a-1}
    \binom{a-i+c-1}{c-1}
    \begin{xy}
      (0,-4)*++!L{\scriptstyle y_i},
      (0,0)*++!L{\scriptstyle y_{a+c-i}},
      (0,4)*++!L{\scriptstyle y_b},
      {(0,-4) \ar @{{*}-{*}} (0,0)},
      {(0,0) \ar @{{*}-{*}} (0,4)}
    \end{xy}- 
    \sum_{i=1}^{c-1}
    \binom{c-i+a-1}{a-1}
    \begin{xy}
      (0,-4)*++!L{\scriptstyle y_i},
      (0,0)*++!L{\scriptstyle y_{a+c-i}},
      (0,4)*++!L{\scriptstyle y_b},
      {(0,-4) \ar @{{*}-{*}} (0,0)},
      {(0,0) \ar @{{*}-{*}} (0,4)}
    \end{xy}\right).
  \end{align*}
  Note that this error term is determined by a given rooted tree and its two vertices cannot be compared in opposite tree order.
\end{eg}

The final computation from the above example can be generalized to the following lemma.
\begin{lem}\label{lem the simple arborification formula}
  Let $Y$ be a $\Ycal$-decorated planar rooted tree with minimal incomparable pair $(a,b)$, and write $Y$ as
  \begin{equation}\label{Y-decorated planar rooted tree}
    Y = B_+^{y_{n_1}} \circ \cdots \circ B_+^{y_{n_m}}(B_+^{y_{n_a}}(F_a)\,B_+^{y_{n_b}}(F_b)\,F),
  \end{equation}
  as in Lemma \ref{lem expression of planar rooted tree}. Then
  \begin{align*}
    & \afk_{P\Xcal}(\sfk^{PN}(Y)) = \\
    & \sum_{i=1}^{n_a} \binom{n_a-i+n_b-1}{n_b-1} \afk_{P\Xcal}(\sfk^{PN}(B_+^{y_{n_1}} \circ \cdots \circ B_+^{y_{n_m}}(B_+^{y_{n_a+n_b-i}}(B_+^{y_i}(F_a)\,F_b)\,F)))\\
    & + \sum_{i=1}^{n_b} \binom{n_b-i+n_a-1}{n_a-1} \afk_{P\Xcal}(\sfk^{PN}(B_+^{y_{n_1}} \circ \cdots \circ B_+^{y_{n_m}}(B_+^{y_{n_a+n_b-i}}(B_+^{y_i}(F_b)\,F_a)\,F))).
  \end{align*}
\end{lem}

\begin{proof}
  The image of \eqref{Y-decorated planar rooted tree} under $\sfk^{PN}$ is the following:
  $$\sfk^{PN}(Y) = 
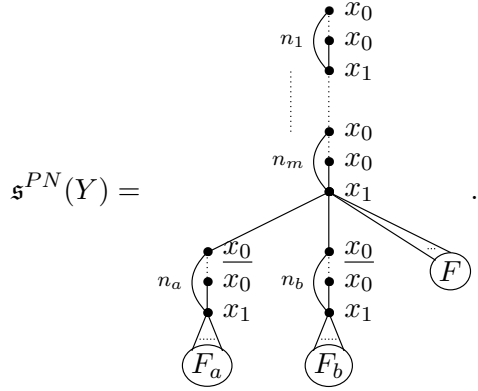
\begin{xy}
    {(0,24) \ar @{{*}.{*}} (0,20)}, (4,24)*{x_0},
    {(0,20) \ar @{{*}-{*}} (0,16)}, (4,20)*{x_0},
    {(0,16) \ar @{{*}.{*}} (0,8)}, (4,16)*{x_1},
    {(-5,16) \ar @{{}.{}} (-5,8)},
    {(0,8) \ar @{{*}.{*}} (0,4)}, (4,8)*{x_0},
    {(0,4) \ar @{{*}-{*}} (0,0)}, (4,4)*{x_0},
    {(0,0) \ar @{{*}-{*}} (0,-8)}, (4,0)*{x_1},
    {(0,0) \ar @{{*}-{}} (14,-9)},
    {(13,-7.5) \ar @{{}.{}} (14,-7.5)},
    {(0,0) \ar @{{*}-{}} (16,-8)},
    {(0,0) \ar @{{*}-{*}} (-16,-8)},
    {(0,-8) \ar @{{*}.{*}} (0,-12)}, (4,-8)*{\underline{x_0}},
    {(-16,-8) \ar @{{*}.{*}} (-16,-12)}, (-12,-8)*{\underline{x_0}},
    {(0,-12) \ar @{{*}-{*}} (0,-16)}, (4,-12)*{x_0},
    {(-16,-12) \ar @{{*}-{*}} (-16,-16)}, (-12,-12)*{x_0},
    {(0,-16) \ar @{{*}-{}} (-2,-21)}, (4,-16)*{x_1},
    {(-1,-19.5) \ar @{{}.{}} (1,-19.5)},
    {(0,-16) \ar @{{*}-{}} (2,-21)},
    {(-16,-16) \ar @{{*}-{}} (-14,-21)}, (-12,-16)*{x_1},
    {(-15,-19.5) \ar @{{}.{}} (-17,-19.5)},
    {(-16,-16) \ar @{{*}-{}} (-18,-21)},
    (-16,-23)*+[o][F-]{F_a},
    (0,-23)*+[o][F-]{F_b},
    (16,-10.5)*+[o][F-]{F},
    {(0,24) \ar @/_2mm/ @{-}_{n_1} (0,16)},
    {(0,8) \ar @/_2mm/ @{-}_{n_m} (0,0)},
    {(0,-8) \ar @/_2mm/ @{-}_{n_b} (0,-16)},
    {(-16,-8) \ar @/_2mm/ @{-}_{n_a} (-16,-16)},
  \end{xy}.$$
  Applying Lemma \ref{lem arborification} (the case of planar rooted forests) to the two vertices decorated by $\underline{x_0}$, we obtain:
  \begin{align*} 
    & \afk_{P\Xcal}(\sfk^{PN}(Y)) = \\
    & \afk_{P\Xcal} \left( \begin{xy}
    {(0,28) \ar @{{*}.{*}} (0,24)}, (4,28)*{x_0},
    {(0,24) \ar @{{*}-{*}} (0,20)}, (4,24)*{x_0},
    {(0,20) \ar @{{*}.{*}} (0,12)}, (4,20)*{x_1},
    {(-5,20) \ar @{{}.{}} (-5,12)},
    {(0,12) \ar @{{*}.{*}} (0,8)}, (4,12)*{x_0},
    {(0,8) \ar @{{*}-{*}} (0,4)}, (4,8)*{x_0},
    {(0,4) \ar @{{*}-{*}} (0,-4)}, (4,4)*{x_1},
    {(0,-4) \ar @{{*}-{*}} (0,-12)}, (4,-4)*{x_0},
    {(0,4) \ar @{{*}-{}} (16,-4)},
    {(13,-3.5) \ar @{{}.{}} (14,-3.5)},
    {(0,4) \ar @{{*}-{}} (14,-5)},
    {(0,-4) \ar @{{*}-{*}} (-16,-12)},
    {(0,-12) \ar @{{*}.{*}} (0,-16)}, (4,-12)*{x_0},
    {(-16,-12) \ar @{{*}.{*}} (-16,-16)}, (-12,-12)*{x_0},
    {(0,-16) \ar @{{*}-{*}} (0,-20)}, (4,-16)*{x_0},
    {(-16,-16) \ar @{{*}-{*}} (-16,-20)}, (-12,-16)*{x_0},
    {(0,-20) \ar @{{*}-{}} (-2,-25)}, (4,-20)*{x_1},
    {(-1,-23.5) \ar @{{}.{}} (1,-23.5)},
    {(0,-20) \ar @{{*}-{}} (2,-25)},
    {(-16,-20) \ar @{{*}-{}} (-14,-25)}, (-12,-20)*{x_1},
    {(-15,-23.5) \ar @{{}.{}} (-17,-23.5)},
    {(-16,-20) \ar @{{*}-{}} (-18,-25)},
    (-16,-27)*+[o][F-]{F_a},
    (0,-27)*+[o][F-]{F_b},
    (16,-6.5)*+[o][F-]{F},
    {(0,28) \ar @/_2mm/ @{-}_{n_1} (0,20)},
    {(0,12) \ar @/_2mm/ @{-}_{n_m} (0,4)},
    {(0,-12) \ar @/_2mm/ @{-}_{n_b} (0,-20)},
    {(-16,-12) \ar @/_2mm/ @{-}_{n_a-1} (-16,-20)},
  \end{xy} \right) + \afk_{P\Xcal} \left( \begin{xy}
    {(0,28) \ar @{{*}.{*}} (0,24)}, (4,28)*{x_0},
    {(0,24) \ar @{{*}-{*}} (0,20)}, (4,24)*{x_0},
    {(0,20) \ar @{{*}.{*}} (0,12)}, (4,20)*{x_1},
    {(-5,20) \ar @{{}.{}} (-5,12)},
    {(0,12) \ar @{{*}.{*}} (0,8)}, (4,12)*{x_0},
    {(0,8) \ar @{{*}-{*}} (0,4)}, (4,8)*{x_0},
    {(0,4) \ar @{{*}-{*}} (0,-4)}, (4,4)*{x_1},
    {(0,-4) \ar @{{*}-{*}} (0,-12)}, (4,-4)*{x_0},
    {(0,4) \ar @{{*}-{}} (16,-4)},
    {(13,-3.5) \ar @{{}.{}} (14,-3.5)},
    {(0,4) \ar @{{*}-{}} (14,-5)},
    {(0,-4) \ar @{{*}-{*}} (-16,-12)},
    {(0,-12) \ar @{{*}.{*}} (0,-16)}, (4,-12)*{x_0},
    {(-16,-12) \ar @{{*}.{*}} (-16,-16)}, (-12,-12)*{x_0},
    {(0,-16) \ar @{{*}-{*}} (0,-20)}, (4,-16)*{x_0},
    {(-16,-16) \ar @{{*}-{*}} (-16,-20)}, (-12,-16)*{x_0},
    {(0,-20) \ar @{{*}-{}} (2,-25)}, (4,-20)*{x_1},
    {(-1,-23.5) \ar @{{}.{}} (1,-23.5)},
    {(0,-20) \ar @{{*}-{}} (-2,-25)},
    {(-16,-20) \ar @{{*}-{}} (-14,-25)}, (-12,-20)*{x_1},
    {(-15,-23.5) \ar @{{}.{}} (-17,-23.5)},
    {(-16,-20) \ar @{{*}-{}} (-18,-25)},
    (-16,-27)*+[o][F-]{F_a},
    (0,-27)*+[o][F-]{F_b},
    (16,-6.5)*+[o][F-]{F},
    {(0,28) \ar @/_2mm/ @{-}_{n_1} (0,20)},
    {(0,12) \ar @/_2mm/ @{-}_{n_m} (0,4)},
    {(0,-12) \ar @/_2mm/ @{-}_{n_b-1} (0,-20)},
    {(-16,-12) \ar @/_2mm/ @{-}_{n_a} (-16,-20)},
  \end{xy} \right).
  \end{align*}
  By repeatedly applying Lemma \ref{lem arborification} to all black dots above $F_a$ and $F_b$, we obtain the following:
  \begin{align*} 
    & \afk_{P\Xcal}(\sfk^{PN}(Y)) = \\
    & \sum_{i=1}^{n_a} \binom{n_a-i+n_b-1}{n_b-1} \afk_{P\Xcal} \left( \begin{xy}
    {(0,28) \ar @{{*}.{*}} (0,24)}, (4,28)*{x_0},
    {(0,24) \ar @{{*}-{*}} (0,20)}, (4,24)*{x_0},
    {(0,20) \ar @{{*}.{*}} (0,12)}, (4,20)*{x_1},
    {(-5,20) \ar @{{}.{}} (-5,12)},
    {(0,12) \ar @{{*}.{*}} (0,8)}, (4,12)*{x_0},
    {(0,8) \ar @{{*}-{*}} (0,4)}, (4,8)*{x_0},
    {(0,4) \ar @{{*}-{*}} (0,0)}, (4,4)*{x_1},
    {(0,4) \ar @{{*}-{}} (14.5,-4.5)},
    {(13.5,-3.5) \ar @{{}.{}} (14.5,-3.5)},
    {(0,4) \ar @{{*}-{}} (17.5,-4.5)},
    {(0,-8) \ar @{{*}-{*}} (-16,-12)},
    {(0,0) \ar @{{*}.{*}} (0,-4)}, (4,0)*{x_0},
    {(-16,-12) \ar @{{*}.{*}} (-16,-16)}, (-12,-12)*{x_0},
    {(0,-4) \ar @{{*}-{*}} (0,-8)}, (4,-4)*{x_0},
    {(-16,-16) \ar @{{*}-{*}} (-16,-20)}, (-12,-16)*{x_0},
    {(0,-8) \ar @{{*}-{}} (-2,-13)}, (4,-8)*{x_1},
    {(-1,-11.5) \ar @{{}.{}} (1,-11.5)},
    {(0,-8) \ar @{{*}-{}} (2,-13)},
    {(-16,-20) \ar @{{*}-{}} (-14,-25)}, (-12,-20)*{x_1},
    {(-15,-23.5) \ar @{{}.{}} (-17,-23.5)},
    {(-16,-20) \ar @{{*}-{}} (-18,-25)},
    (-16,-27)*+[o][F-]{F_a},
    (0,-15)*+[o][F-]{F_b},
    (16,-6.5)*+[o][F-]{F},
    {(0,28) \ar @/_2mm/ @{-}_{n_1} (0,20)},
    {(0,12) \ar @/_2mm/ @{-}_{n_m} (0,4)},
    {(0,0) \ar @/_2mm/ @{-}_{n_b+n_a-i} (0,-8)},
    {(-16,-12) \ar @/_2mm/ @{-}_{i} (-16,-20)},
    \end{xy} \right)\\
    & + \sum_{i=1}^{n_b} \binom{n_b-i+n_a-1}{n_a-1} \afk_{P\Xcal} \left( \begin{xy}
    {(0,28) \ar @{{*}.{*}} (0,24)}, (4,28)*{x_0},
    {(0,24) \ar @{{*}-{*}} (0,20)}, (4,24)*{x_0},
    {(0,20) \ar @{{*}.{*}} (0,12)}, (4,20)*{x_1},
    {(-5,20) \ar @{{}.{}} (-5,12)},
    {(0,12) \ar @{{*}.{*}} (0,8)}, (4,12)*{x_0},
    {(0,8) \ar @{{*}-{*}} (0,4)}, (4,8)*{x_0},
    {(0,4) \ar @{{*}-{*}} (0,0)}, (4,4)*{x_1},
    {(0,4) \ar @{{*}-{}} (14.5,-4.5)},
    {(13.5,-3.5) \ar @{{}.{}} (14.5,-3.5)},
    {(0,4) \ar @{{*}-{}} (17.5,-4.5)},
    {(0,-8) \ar @{{*}-{}} (-17.5,-16.5)},
    {(-13.5,-15.5) \ar @{{}.{}} (-14.5,-15.5)},
    {(0,-8) \ar @{{*}-{}} (-14.5,-16.5)},
    {(0,-12) \ar @{{*}.{*}} (0,-16)}, (4,-12)*{x_0},
    {(0,0) \ar @{{*}.{*}} (0,-4)}, (4,0)*{x_0},
    {(0,-16) \ar @{{*}-{*}} (0,-20)}, (4,-16)*{x_0},
    {(0,-4) \ar @{{*}-{*}} (0,-8)}, (4,-4)*{x_0},
    {(0,-20) \ar @{{*}-{}} (2,-25)}, (4,-20)*{x_1},
    {(1,-23.5) \ar @{{}.{}} (-1,-23.5)},
    {(0,-20) \ar @{{*}-{}} (-2,-25)},
    {(0,-8) \ar @{{*}-{*}} (0,-12)}, (4,-8)*{x_1},
    (-16,-19)*+[o][F-]{F_a},
    (0,-27)*+[o][F-]{F_b},
    (16,-6.5)*+[o][F-]{F},
    {(0,28) \ar @/_2mm/ @{-}_{n_1} (0,20)},
    {(0,12) \ar @/_2mm/ @{-}_{n_m} (0,4)},
    {(0,-12) \ar @/_2mm/ @{-}_{i} (0,-20)},
    {(0,0) \ar @/_2mm/ @{-}_{n_a+n_b-i} (0,-8)},
    \end{xy} \right).
  \end{align*}
  Rewriting this equation using $\sfk^{PN}$ gives the desired identity, completing the proof.
\end{proof}

We define the error term of a planar rooted tree using its order relation as follows:
\begin{defn}\label{defn error term}
  Let $Y$ be a $\Ycal$-decorated planar rooted tree presented as \eqref{Y-decorated planar rooted tree} and $(a, b)$ the minimal incomparable pair of $Y$. Its {\it error term} $Y^e$ of a $\Ycal$-decorated planar rooted tree $Y$ is defined by
  \begin{align*}
    Y^e =& B_+^{y_{n_1}} \circ \cdots \circ B_+^{y_{n_m}}(B_+^{y_{n_a+n_b}}(F_a\,F_b)\,F) \\
    & - \sum_{i=1}^{n_a-1} \binom{n_a-i+n_b-1}{n_b-1} B_+^{y_{n_1}} \circ \cdots \circ B_+^{y_{n_m}}(B_+^{y_{n_a+n_b-i}}(B_+^{y_i}(F_a)\,F_b)\,F)\\
    & - \sum_{i=1}^{n_b-1} \binom{n_b-i+n_a-1}{n_a-1} B_+^{y_{n_1}} \circ \cdots \circ B_+^{y_{n_m}}(B_+^{y_{n_a+n_b-i}}(B_+^{y_i}(F_b)\,F_a)\,F),
  \end{align*}
  which is an element in $\Hcal_{NBCK}^{P\Ycal}$. If $Y$ is a ladder tree, then there is no minimal incomparable pair of $Y$, and in this case, $Y^e$ is zero.
\end{defn}

We have the following lemma concerning the error term.

\begin{lem}\label{lem error term}
  For any $\Ycal$-decorated planar rooted tree $Y$, the following identity holds:
  $$\afk_{P\Xcal}(\sfk^{PN}(Y) + \sfk^{PN}(Y^e)) = \afk_{P\Xcal}(\sfk^{PN}(Y_{a+b}) + \sfk^{PN}(Y_{b}^{a}) + \sfk^{PN}(Y_{a}^{b})),$$
  where
  \begin{align*}
  Y_{a+b} := & B_+^{y_{n_1}} \circ \cdots \circ B_+^{y_{n_m}}(B_+^{y_{n_a+n_b}}(F_a\,F_b)\,F)\\
  Y_{b}^{a} := & B_+^{y_{n_1}} \circ \cdots \circ B_+^{y_{n_m}}(B_+^{y_{n_a}}(B_+^{y_{n_b}}(F_b)\,F_a)\,F)\\
  Y_{a}^{b} := & B_+^{y_{n_1}} \circ \cdots \circ B_+^{y_{n_m}}(B_+^{y_{n_b}}(B_+^{y_{n_a}}(F_a)\,F_b)\,F).
  \end{align*}
\end{lem}

\begin{proof}
  By Definition \ref{defn error term}, we add the error term $\afk_{P\Xcal}(\sfk^{PN}(Y^e))$ to both sides of the identity in Lemma \ref{lem the simple arborification formula}. Then the left-hand side becomes
  $$\afk_{P\Xcal}(\sfk^{PN}(Y)) + \afk_{P\Xcal}(\sfk^{PN}(Y^e)),$$
  and in the summation on the right-hand side, all terms cancel except for those with $i = n_a$ and $i = n_b$. This implies that the right-hand side becomes
  $$\afk_{P\Xcal}(\sfk^{PN}(Y_{a+b})) + \afk_{P\Xcal}(\sfk^{PN}(Y_{b}^{a})) + \afk_{P\Xcal}(\sfk^{PN}(Y_{a}^{b})),$$
  which complete the proof.
\end{proof}

To construct the map $\sfk^{PT}$, we require the process tree $\pr(Y)$ of contracting arborification of a $\Ycal$-decorated planar rooted tree $Y$.
\begin{defn}\label{defn process tree of contracting arborification}
  The {\it process tree $\pr(Y) = (\pr(Y), \preceq_{\pr(Y)}, \delta_{\pr(Y)})$ of the contracting arborification} $\afk_{P\Ycal}$ of a $\Ycal$-decorated planar rooted tree $Y$ is a $\Hcal_{NBCK}^{P\Ycal}$-decorated planar rooted tree defined recursively by
  $$\pr(Y) = B_+^{Y}(\pr(Y_{a+b})\, \pr(Y_{b}^{a})\, \pr(Y_{a}^{b})),$$
  where $(a, b)$ is the minimal incomparable pair of $Y$ and $Y_{a+b}, Y_{b}^{a}, Y_{a}^{b}$ are defined as in Lemma \ref{lem error term}.
\end{defn}

\begin{eg}
  Consider the $\Ycal$-decorated planar rooted tree
  $$Y = \begin{xy}
    {(-4,-2)*++!R{\scriptstyle \alpha, y_a} \ar @{{*}-{*}} (0,2)*++!D{\scriptstyle \beta, y_b}},
    {(0,2) \ar @{{*}-{*}} (4,-2)*++!L{\scriptstyle \gamma, y_c}}
  \end{xy}.$$
  The three vertices $\alpha, \beta, \gamma$ of $Y$ are decorated by $y_a, y_b, y_c \in \Ycal$, respectively. To describe the process tree $\pr(Y)$ of the contracting arborification of $Y$, we compute the following first.
  $$Y_{\alpha+\gamma} = \begin{xy}
    (0,-2)*++!L{\scriptstyle y_{a+c}},
    (0,2)*++!L{\scriptstyle y_{b}},
    {(0,-2) \ar @{{*}-{*}} (0,2)}
  \end{xy}, Y_{\gamma}^{\alpha} = \begin{xy}
    (0,-4)*++!L{\scriptstyle y_{c}},
    (0,0)*++!L{\scriptstyle y_{a}},
    (0,4)*++!L{\scriptstyle y_{b}},
    {(0,-4) \ar @{{*}-{*}} (0,0)},
    {(0,0) \ar @{{*}-{*}} (0,4)}
  \end{xy}, Y_{\alpha}^{\gamma} = \begin{xy}
    (0,-4)*++!L{\scriptstyle y_{a}},
    (0,0)*++!L{\scriptstyle y_{c}},
    (0,4)*++!L{\scriptstyle y_{b}},
    {(0,-4) \ar @{{*}-{*}} (0,0)},
    {(0,0) \ar @{{*}-{*}} (0,4)}
  \end{xy}.$$
  Thus, the process tree $\pr(Y)$ is given by
  $$\pr(Y) = \begin{xy}
    {(-6,-2)*++!U{\scriptstyle Y_{\alpha+\gamma}} \ar @{{*}-{*}} (0,2)*++!D{\scriptstyle Y}},
    {(0,2) \ar @{{*}-{*}} (6,-2)*++!U{\scriptstyle Y_{\alpha}^{\gamma}}},
    {(0,2) \ar @{{*}-{*}} (0,-2)*++!U{\scriptstyle Y_{\gamma}^{\alpha}}}
  \end{xy}.$$
\end{eg}

We use the process tree to define the following map:
\begin{defn}\label{defn map sPT}
  The linear map $\phi : \Hcal_{NBCK}^{P\Ycal} \rightarrow \Hcal_{NBCK}^{P\Ycal}$ is defined by
  $$\phi(Y) = Y + \sum_{v\in V(pr(Y))}\left(\delta_{\pr(Y)}(v)\right)^e.$$
  Here $\delta^e$ means the error term appearing in Definition \ref{defn error term}.\\
  The map $\sfk^{PT} : \Hcal_{NBCK}^{P\Ycal} \rightarrow \Hcal_{NBCK}^{P\Xcal}$ is defined by
  $$\sfk^{PT}(Y) = \sfk^{PN}\circ \phi(Y).$$
\end{defn}

Since this function involves only the necessary error terms, we state the following lemma for it.
\begin{lem}\label{lem the map sPT}
  The following diagram is commutative.
  $$\begin{tikzcd}
   \Hcal_{NBCK}^{P\Ycal} \arrow[r, "\sfk^{PT}"]  & \Hcal_{NBCK}^{P\Xcal} \arrow[d, "\afk_{P\Xcal}"] \\
  \QQ\langle\Ycal\rangle \arrow[r, "\sfk"] \arrow[u, "\ell_{P\Ycal}"'] & \QQ\langle\Xcal\rangle               
  \end{tikzcd}$$
\end{lem}

\begin{proof}
  Let $y_{n_1}\cdots y_{n_m}$ be an element in $\QQ\langle \Ycal \rangle$.
  Since $\ell_{P\Ycal}(y_{n_1}\cdots y_{n_m})$ is a ladder tree, we have
  $$\pr(\ell_{P\Ycal}(y_{n_1}\cdots y_{n_m})) = B_+^{\ell_{P\Ycal}(y_{n_1}\cdots y_{n_m})}(\emptyset).$$
  Hence, the error term of $\ell_{P\Ycal}(y_{n_1}\cdots y_{n_m})$ vanishes. We obtain
  $$\afk_{P\Xcal} \circ \sfk^{PT} \circ \ell_{P\Ycal}(y_{n_1}\cdots y_{n_m}) = \afk_{P\Xcal} \circ \sfk^{PN} \circ \ell_{P\Ycal}(y_{n_1}\cdots y_{n_m}).$$
  By Lemma \ref{lem map sPN}, 
  $$\afk_{P\Xcal} \circ \sfk^{PN} \circ \ell_{P\Ycal} = \sfk,$$
  we obtain
  $$\afk_{P\Xcal} \circ \sfk^{PN} \circ \ell_{P\Ycal}(y_{n_1}\cdots y_{n_m}) = \sfk(y_{n_1}\cdots y_{n_m}),$$
  which completes the proof.
\end{proof}

The following result is the main theorem of this paper.
\begin{thm}\label{thm Main theorem}
  The following diagram is commutative.
  $$\begin{tikzcd}
   & \Hcal_{NBCK}^{P\Ycal} \arrow[rd, "\sfk^{PN}"] & \\
  \Hcal_{NBCK}^{P\Ycal} \arrow[rr, "\sfk^{PT}"] \arrow[ru, "\phi"] \arrow[d, "\afk_{P\Ycal}"'] & & \Hcal_{NBCK}^{P\Xcal} \arrow[d, "\afk_{P\Xcal}"] \\
  \QQ\langle\Ycal\rangle \arrow[rr, "\sfk"]               &  & \QQ\langle\Xcal\rangle               
  \end{tikzcd}$$
\end{thm}

\begin{proof}
  Let $Y$ be a $\Ycal$-decorated planar rooted tree. By Lemma \ref{lem the map sPT}, the diagram commutes when $Y$ is a ladder tree, that is,
  $$\afk_{P\Xcal} \circ \sfk^{PT}(Y) = \sfk \circ \afk_{P\Ycal}(Y).$$
  To prove the general case, we compute $\sfk^{PT}(Y)$ and obtain
  \begin{align*}
    \sfk^{PT}(Y) = & \sfk^{PN}\circ \phi(Y) = \sfk^{PN}\left(Y + \sum_{v\in V(\pr(Y))}\left(\delta_{\pr(Y)}(v)\right)^e\right) \\
    = & \sfk^{PN}\left(Y + \left(\delta_{\pr(Y)}(r)\right)^e + \sum_{v\in V(\pr(Y))\setminus\{r\}}\left(\delta_{\pr(Y)}(v)\right)^e\right),
  \end{align*}
  where $V(\pr(Y))$ is the vertex set of $\pr(Y)$ and $r$ is the root of $\pr(Y)$.
  By Lemma \ref{lem error term}, we have the identity:
  $$\afk_{P\Xcal}\left(\sfk^{PN}\left(Y + \left(\delta_{\pr(Y)}(r)\right)^e\right)\right) = \afk_{P\Xcal}\left(\sfk^{PN}(Y_{a+b}) + \sfk^{PN}(Y_{b}^{a}) + \sfk^{PN}(Y_{a}^{b})\right),$$
  because $\delta_{\pr(Y)}(r) = Y$.
  On the other hand, by Definition \ref{defn process tree of contracting arborification}, the remaining vertices decompose as
  $$V(\pr(Y))\setminus\{r\} = V(\pr(Y_{a+b})) \cup V(\pr(Y_{b}^{a})) \cup V(\pr(Y_{a}^{b})),$$
  and the decoration maps are compatible:
  $$\delta_{\pr(Y_{a+b})} = \delta_{\pr(Y)}|_{\pr(Y_{a+b})},\ \delta_{\pr(Y_{b}^{a})} = \delta_{\pr(Y)}|_{\pr(Y_{b}^{a})},\ \delta_{\pr(Y_{a}^{b})} = \delta_{\pr(Y)}|_{\pr(Y_{a}^{b})}.$$
  From the above, we obtain
  \begin{align*}
    \afk_{P\Xcal} \circ \sfk^{PT}(Y) = & \afk_{P\Xcal}\left(\sfk^{PN}\left(Y + \left(\delta_{\pr(Y)}(r)\right)^e\right)\right) + \afk_{P\Xcal}\left(\sfk^{PN}\left(\sum_{v\in V(\pr(Y))\setminus\{r\}}\left(\delta_{\pr(Y)}(v)\right)^e\right)\right)\\
    = & \afk_{P\Xcal}\left(\sfk^{PN}\left(Y_{a+b} + Y_{b}^{a} + Y_{a}^{b}\right)\right)\\
    & + \afk_{P\Xcal}\left(\sfk^{PN}\left(\sum_{v\in V(\pr(Y_{a+b})) \cup V(\pr(Y_{b}^{a})) \cup V(\pr(Y_{a}^{b}))}\left(\delta_{\pr(Y_{a+b})}(v)\right)^e\right)\right)\\
    = & \afk_{P\Xcal} \left ( \sfk^{PN}\left(Y_{a+b} + \sum_{v\in V(\pr(Y_{a+b}))}\left(\delta_{\pr(Y_{a+b})}(v)\right)^e \right) \right) \\
    & + \afk_{P\Xcal} \left ( \sfk^{PN}\left(Y_{b}^{a} + \sum_{v\in V(\pr(Y_{b}^{a}))}\left(\delta_{\pr(Y_{b}^{a})}(v)\right)^e \right) \right) \\
    & + \afk_{P\Xcal} \left ( \sfk^{PN}\left(Y_{a}^{b} + \sum_{v\in V(\pr(Y_{a}^{b}))}\left(\delta_{\pr(Y_{a}^{b})}(v)\right)^e \right) \right) \\
    = & \sfk^{PT}(Y_{a+b}) + \sfk^{PT}(Y_{b}^{a}) + \sfk^{PT}(Y_{a}^{b}).
  \end{align*}
  We can recursively use the above equation to obtain
  $$\sfk^{PT}(Y) = \sum_{v\in \text{leaf}(\pr(Y))}\sfk^{PT}(\delta_{\pr(Y)}(v)).$$
  Since $\delta_{\pr(Y)}(v)$ is a ladder tree for any $v\in \text{leaf}(\pr(Y))$, we get
  \begin{align*}
    \afk_{P\Xcal} \circ \sfk^{PT}(Y) =  & \sum_{v\in \text{leaf}(\pr(Y))}\afk_{P\Xcal} \circ \sfk^{PT}(\delta_{\pr(Y)}(v)) \\
    = & \sum_{v\in \text{leaf}(\pr(Y))} \sfk \circ \afk_{P\Ycal}(\delta_{\pr(Y)}(v)) = \sfk \circ \afk_{P\Ycal}(Y),
  \end{align*}
  which proves this theorem.
\end{proof}

We define the following map to answer Manchon's question.
\begin{defn}\label{defn map sT}
  Let $Y$ be a $\Ycal$-decorated rooted tree and $A_Y$ the set of all total order relations that make $Y$ a planar rooted tree.
  The cardinality of the set $A_Y$ is given by 
  $$deg(r)\times \prod_{v\in V(Y)\setminus \text{leaf}(Y)\setminus\{r\}}(deg(v)-1).$$
  The section $\beta_\Ycal$ of $\hat{\alpha}_{\Ycal} : \Hcal_{NBCK}^{P\Ycal} \rightarrow \Hcal_{BCK}^{\Ycal}$ is defined by
  $$\beta_\Ycal(Y) \coloneqq \frac{1}{|A_Y|}\sum_{\alpha \in |A_Y|}Y_\alpha,$$
  where $Y_\alpha$ is the $\Ycal$-decorated planar rooted tree obtained by equipping the $\Ycal$-decorated rooted tree $Y$ with the total order relation $\alpha\in A_Y$. We define the map
  $$\sfk^T \coloneqq \hat{\alpha}_{\Xcal} \circ \sfk^{PT} \circ \beta_\Ycal.$$
\end{defn}

Then we see that the diagram \eqref{diagram} with the above $\sfk^T$ is commutative:
\begin{cor}
  The following diagram is commutative:
  $$\begin{tikzcd}
  & \Hcal_{NBCK}^{P\Ycal} \arrow[rd, "\sfk^{PN}"] & \\
  \Hcal_{NBCK}^{P\Ycal} \arrow[ru, "\phi"] \arrow[dd, "\afk_{P\Ycal}"', bend right] \arrow[rr, "\sfk^{PT}"] & & \Hcal_{NBCK}^{P\Xcal} \arrow[d, "\hat{\alpha}_{\Xcal}"'] \arrow[dd, "\afk_{P\Xcal}", bend left] \\
  \Hcal_{BCK}^{\Ycal} \arrow[d, "\afk_{\Ycal}"] \arrow[rr, "\sfk^T"] \arrow[u, "\beta_\Ycal"']  &   & \Hcal_{BCK}^{\Xcal} \arrow[d, "\afk_{\Xcal}"']\\
  \QQ\langle \Ycal \rangle \arrow[rr, "\sfk"]  &    & \QQ\langle \Xcal \rangle                                           
  \end{tikzcd}$$
\end{cor}

\begin{proof}
  By Definition \ref{defn map sT}, the rectangle in the center is commutative. By Definition \ref{defn arborification}, the left triangle and the right triangle commute. By Theorem \ref{thm Main theorem}, the outer pentagon commutes. From the above, all regions in the diagram commute. Hence, our claim follows.
\end{proof}
$ $\\
\textbf{Acknowledgements} I would like to thank Professor Furusho for the helpful advice and kind support during the writing of this paper.

\end{document}